\documentclass[a4paper,11pt]{article}
\usepackage{amssymb,amsthm,amsmath,enumerate}

    \newtheorem{theorem}{Theorem}[section]
    \theoremstyle{definition}
    \newtheorem{example}[theorem]{Example}

    \newcommand{\is}{ \mathrm{IS} }
    \newcommand{\dc}{ \mathrm{DC} }
    \newcommand{\ic}{ \mathrm{IC} }
    \newcommand{\ub}{ \mathrm{ub} }
    \newcommand{\infi}{ \infty }

\begin{document}

\title{On Maximal Chains of Systems of Word Equations%
    \thanks{Supported by the Academy of Finland under grant 121419}
}

\author{Juhani Karhum\"{a}ki and Aleksi Saarela \\
    Department of Mathematics and Statistics \\
    University of Turku,
    20014 Turku, Finland \\
}

\date{
    Originally published in 2011\\
    Note added in 2015
}

\maketitle

\begin{abstract}
We consider systems of word equations and their solution sets. We
discuss some fascinating properties of those, namely the size of
a maximal independent set of word equations, and proper chains of
solution sets of those. We recall the basic results, extend some
known results and formulate several fundamental problems of the
topic.

Keywords: word equations, independent systems, solution chains
\end{abstract}

\section{Introduction}

Theory of word equations is a fundamental part of combinatorics on
words. It is a challenging topic of its own which has a number of
connections and applications, e.g., in pattern unification and group
representations. There have also been several fundamental
achievements in the theory over the last few decades.

Decidability of the existence of a solution of a given word equation
is one fundamental result due to Makanin \cite{Ma77}. This is in
contrast to the same problem on Diophantine equations, which is
undecidable \cite{Ma70}. Although the complexity of the above
\emph{satisfiability problem} for word equations is not known, a
nontrivial upper bound has been proved: it is in PSPACE \cite{Pl04}.

Another fundamental property of word equations is the so-called
\emph{Ehrenfeucht compactness property}. It guarantees that any
system of word equations is equivalent to some of its finite
subsystems. The proofs (see \cite{AlLa85ehrenfeucht} and \cite{Gu86}) are based
on a transformation of word equations into Diophantine equations and
then an application of Hilbert's basis theorem. Although we have
this finiteness property, we do not know any upper bound, if it exists,
for the size of an equivalent subsystem in terms of the number of
unknowns. And this holds even in the case of three unknown systems
of equations. In free monoids an equivalent formulation of the
compactness property is that each \emph{independent} system of word
equations is finite, independent meaning that the system is not
equivalent to any of its proper subsystems. We analyze in this paper
the size of the maximal independent systems of word equations.

As a related problem we define the notion of \emph{decreasing
chains} of word equations. Intuitively, this asks how long chains of
word equations exist such that the set of solutions always properly
diminishes when a new element of the chain is taken into the system.
Or more intuitively, how many proper constraints we can define such
that each constraint reduces the set of words satisfying these
constraints. It is essentially the above compactness property which
guarantees that these chains are finite.

Another fundamental property of word equations is the result of
Hmelevskii \cite{Hm71} stating that for each word equation with
three unknowns its solution set is \emph{finitely parameterizable}.
This result is not directly related to our considerations, but its
intriguity gives, we believe, a strong explanation and support to
our view that our open problems, even the simplest looking ones, are
not trivial. Hmelevskii's argumentation is simplified in the extended abstract
\cite{KaSa08dlt}, and used in
\cite{Sa09dlt} to show that the satisfiability problem for three
unknown equations is in NP.
A full version of these two conference articles has been submitted \cite{KaSa15}.

The goal of this note is to analyze the above maximal independent
systems of equations and maximal decreasing chains of word
equations, as well as search for their relations. An essential part
is to propose open problems on this area. The most fundamental
problem asks whether the maximal independent system of word
equations with $n$ unknowns is bounded by some  function of $n$.
Amazingly, the same problem is open for three unknown equations,
although we do not know larger than three equation systems in this
case.

\section{Systems and Chains of Word Equations}

The topics of this paper are independent systems and chains of
equations in semigroups. We are mostly interested in free monoids;
in this case the equations are constant-free word equations. We
present some questions about the sizes of such systems and chains,
state existing results, give some new ones, and list open problems.

Let $S$ be a semigroup and $\Xi$ be an alphabet of variables. We
consider equations $U=V$, where $U,V \in \Xi^+$. A morphism $h:
\Xi^+ \to S$ is a \emph{solution} of this equation if $h(U) =
h(V)$. (If $S$ is a monoid, we can use $\Xi^*$ instead of $\Xi^+$.)

A system of equations is \emph{independent}
if it is not equivalent to any of its proper subsystems.
In other words, equations $E_i$
form an independent system of equations if for every $i$ there is a
morphism $h_i$ which is not a solution of $E_i$ but which is a
solution of all the other equations. This definition works for both
finite and infinite systems of equations.

We define \emph{decreasing chains} of equations. A finite sequence
of equations $E_1, \dots, E_m$ is a decreasing chain if for every
$i \in \{0, \dots, m-1\}$ the system $E_1, \dots, E_i$ is
inequivalent to the system $E_1, \dots, E_{i+1}$. An infinite
sequence of equations $E_1, E_2, \dots$ is a decreasing chain if
for every $i \geq 0$ the system $E_1, \dots, E_i$
is inequivalent to the system $E_1, \dots, E_{i+1}$.

Similarly we define \emph{increasing chains} of equations. A
sequence of equations $E_1, \dots, E_m$ is an increasing chain if
for every $i \in \{1, \dots, m\}$ the system $E_i, \dots, E_m$ is
inequivalent to the system $E_{i+1}, \dots, E_m$. An infinite
sequence of equations $E_1, E_2, \dots$ is an increasing chain if
for every $i \geq 1$ the system $E_i, E_{i+1}, \dots$ is
inequivalent to the system $E_{i+1}, E_{i+2}, \dots$.

Now $E_1, \dots, E_m$ is an increasing chain if and only if $E_m,
\dots, E_1$ is a decreasing chain. However, for infinite chains
these concepts are essentially different. Note that a chain can be
both decreasing and increasing, for example, if the equations form an
independent system.

We will consider the \emph{maximal} sizes of independent systems of
equations and chains of equations. If the number of unknowns is $n$,
then the maximal size of an independent system is denoted by
$\is(n)$. We use two special symbols $\ub$ and $\infi$ for the
infinite cases: if there are infinite independent systems, then
$\is(n) = \infi$, and if there are only finite but unboundedly large
independent systems, then $\is(n) = \ub$. We extend the order
relation of numbers to these symbols: $k < \ub < \infi$ for every
integer $k$. Similarly the maximal size of a decreasing chain is
denoted by $\dc(n)$, and the maximal size of an increasing chain by
$\ic(n)$.

Often we are interested in the finiteness of $\dc(n)$, or its
asymptotic behaviour when $n$ grows. However, if we are interested
in the exact value of $\dc(n)$, then some technical remarks about
the definition are in order. First, the case $i=0$ means that there
is a solution which is not a solution of the first equation $E_1$;
that is, $E_1$ cannot be a trivial equation like $U = U$. If this
condition was removed, then we could always add a trivial equation
in the beginning, and $\dc(n)$ would be increased by one. Second, we
could add the requirement that there must be a solution which is a
solution of all the equations $E_1, \dots, E_m$, and the definition
would remain the same in the case of free monoids. However, if we
consider free semigroups, then this addition would change the
definition, because then $E_m$ could not be an equation with no
solutions, like $xx = x$ in free semigroups. This would decrease
$\dc(n)$ by one.

\section{Relations Between Systems and Chains}

Independent systems of equations are a well-known topic (see, e.g.,
\cite{HaKaPl02}). Chains of equations have been studied less, so we
prove here some elementary results about them. The following theorem
states the most basic relations between $\is$, $\dc$ and $\ic$.

\begin{theorem} \label{thm:basic}
For every $n$,
\begin{math}
    \is(n) \leq \dc(n), \ic(n) .
\end{math}
If $\dc(n) < \ub$ or $\ic(n) < \ub$, then
\begin{math}
    \dc(n) = \ic(n).
\end{math}
\end{theorem}
\begin{proof}
Every independent system of equations is also a decreasing and
increasing chain of equations, regardless of the order of the
equations. This means that
\begin{math}
    \is(n) \leq \dc(n), \ic(n) .
\end{math}

A finite sequence of equations is a decreasing chain if and only if
the reverse of this sequence is an increasing chain. Thus
\begin{math}
    \dc(n) = \ic(n),
\end{math}
if $\dc(n) < \ub$ or $\ic(n) < \ub$.
\end{proof}

A semigroup has the \emph{compactness property} if every system of
equations has an equivalent finite subsystem. Many results on the
compactness property are collected in \cite{HaKaPl02}. In terms of
chains, the compactness property turns out to be equivalent to the
property that every decreasing chain is finite.

\begin{theorem} \label{thm:cp_dc}
A semigroup has the compactness property if and only if $\dc(n) \leq
\ub$ for every $n$.
\end{theorem}
\begin{proof}
Assume first that the compactness property holds. Let $E_1, E_2,
\dots$ be an infinite decreasing chain of equations. As a system of
equations, it is equivalent to some finite subsystem $E_{i_1},
\dots, E_{i_k}$, where $i_1 < \dots < i_k$. But now $E_1, \dots
E_{i_k}$ is equivalent to $E_1, \dots, E_{i_k + 1}$. This is a
contradiction.

Assume then that $\dc(n) \leq \ub$. Let $E_1, E_2, \dots$ be an
infinite system of equations. If there is an index $N$ such that
$E_1, \dots, E_i$ is equivalent to $E_1, \dots, E_{i+1}$ for all
$i \geq N$, then the whole system is equivalent to $E_1, \dots,
E_N$. If there is no such index, then let $i_1 < i_2 < \dots$ be all
indexes such that $E_1, \dots E_{i_k}$ is not equivalent to $E_1,
\dots, E_{i_k + 1}$. But then $E_{i_1}, E_{i_2}, \dots$ is an
infinite decreasing chain, which is a contradiction.
\end{proof}

The next example shows that the values of $\is$, $\dc$ and $\ic$ can
differ significantly.

\begin{example}
We give an example of a monoid where $\is(1) = 1$, $\dc(1) = \ub$
and $\ic(1) = \infi$. The monoid is
\begin{equation*}
    \langle a_1, a_2, \dots
    \ | \ a_i a_j = a_j a_i , \ a_i^{i+1} = a_i^i \rangle .
\end{equation*}
Now every equation on one unknown is of the form $x^i = x^j$. If
$i<j$, then this is equivalent to $x^i = x^{i+1}$. So all
nontrivial equations are, up to equivalence,
\begin{equation*}
    x = 1, \ x^2 = x, \ x^3 = x^2, \ \dots ,
\end{equation*}
and these have strictly increasing solution sets. Thus $\ic(1) =
\infi$, $\dc(1) = \ub$ and $\is(1) = 1$.
\end{example}

\section{Free Monoids}

From now on we will consider free monoids and semigroups. The bounds
related to free monoids are denoted by $\is$, $\dc$ and $\ic$, and
the bounds related to free semigroups, by $\is_+$, $\dc_+$ and
$\ic_+$.

We give some definitions related to word equations and make some
easy observations about the relations between maximal sizes of
independent systems and chains, assuming these are finite.

A solution $h$ is \emph{periodic} if there exists a $t \in S$
such that every $h(x)$, where $x \in \Xi$, is a power of $t$. Otherwise
$h$ is \emph{nonperiodic}. An equation $U=V$ is \emph{balanced} if
every variable occurs as many times in $U$ as in $V$.

The maximal size of an independent system in a free monoid having a
nonperiodic solution is denoted by $\is'(n)$. The maximal size of a
decreasing chain having a nonperiodic solution is denoted by
$\dc'(n)$. Similar notation can be used for free semigroups.

Every independent system of equations $E_1, \dots, E_m$ is also a
chain of equations, regardless of the order of the equations. If the
system has a nonperiodic solution, then we can add an equation that
forces the variables to commute. If the equations in the system are
also balanced, then we can add equations $x_i = 1$ for all variables
$x_1, \dots, x_n$, and thus get a chain of length $m+n+1$. If they
are not balanced, then we can add at least one of these equations.

In all cases we obtain the inequalities $\is'(n) \leq \is(n) \leq
\is'(n) + 1$ and $\dc'(n) + 2 \leq \dc(n) \leq \dc'(n) + n + 1$, as
well as $\is(n) + 1 \leq \dc(n)$ and $\is'(n) \leq \dc'(n)$. In the
case of free semigroups we derive similar inequalities. Thus $\is'$
and $\dc'$ are basically the same as $\is$ and $\dc$, if we are only
interested in their finiteness or asymptotic growth.

It was conjectured by Ehrenfeucht in a language theoretic setting
that the compactness property holds for free monoids. This
conjecture was reformulated in terms of equations in \cite{CuKa83},
and it was proved independently by Albert and Lawrence \cite{AlLa85ehrenfeucht}
and by Guba \cite{Gu86}.

\begin{theorem} \label{thm:compactness}
$\dc(n) \leq \ub$, and hence also $\is(n) \leq \ub$.
\end{theorem}

The proofs are based on Hilbert's basis theorem. The compactness
property means that $\dc(n) \leq \ub$ for every $n$. No better upper
bounds are known, when $n > 2$. Even the seemingly simple question
about the size of $\is'(3)$ is still completely open; the only thing
that is known is that $2 \leq \is'(3) \leq \ub$. The lower bound is
given by the example $xyz=zyx, xyyz=zyyx$.

\section{Three and Four Unknowns}

The cases of three and four variables have been studied in
\cite{Cz08}. The article gives examples showing that $\is'_+(3) \geq
2$, $\dc_+(3) \geq 6$, $\is'_+(4) \geq 3$ and $\dc_+(4) \geq 9$. We
are able to give better bounds for $\dc_+(3)$ and $\dc(4)$.

First we assume that there are three unknowns $x$, $y$, $z$. There
are trivial examples of independent systems of three equations, for
example, $x^2=y, y^2=z, z^2=x$, so $\is_+(3) \geq 3$. There are also
easy examples of independent pairs of equations having a nonperiodic
solution, like $xyz=zyx, xyyz=zyyx$, so $\is'_+(3) \geq 2$.
Amazingly, no other bounds are known for $\is_+(3)$, $\is'_+(3)$,
$\is(3)$ or $\is'(3)$.

{\allowdisplaybreaks The following chain of equations shows that
$\dc(3) \geq 7$:
\begin{align*}
    xyz &= zxy,&
        x&=a, \ y=b, \ z=abab\\
    xy xzy z &= z xzy xy,&
        x&=a, \ y=b, \ z=ab\\
    xz &= zx,&
        x&=a, \ y=b, \ z=1\\
    xy &= yx,&
        x&=a, \ y=a, \ z=a\\
    x &= 1,&
        x&=1, \ y=b, \ z=a\\
    y &= 1,&
        x&=1, \ y=1, \ z=a\\
    z &= 1,&
        x&=1, \ y=1, \ z=1 .
\end{align*}
Here the second column gives a solution which is not a solution of
the equation on the next row but is a solution of all the preceding
equations. Also $\dc_+(3) \geq 7$, as shown by the chain
\begin{align*}
    xxyz &= zxyx,&
        &x=a, \ y=b, \ z=aabaaba\\
    xxyx zy z &= z zy xxyx,&
        &x=a, \ y=b, \ z=aaba\\
    xz &= zx,&
        &x=a, \ y=b, \ z=a\\
    xy &= yx,&
        &x=a, \ y=aa, \ z=a\\
    x &= y,&
        &x=a, \ y=a, \ z=aa\\
    x &= z,&
        &x=a, \ y=a, \ z=a\\
    xx &= x,&
        &\text{no solutions}.
\end{align*}}

If there are three variables, then every independent pair of
equations having a nonperiodic solution consists of balanced
equations (see \cite{HaNo03}). It follows that $\is'(3) + 4 \leq
\dc(3)$. There are also some other results about the structure of
equations in independent systems on three unknowns (see
\cite{CzKa07} and \cite{CzPl09}).

{\allowdisplaybreaks If we add a fourth unknown $t$, then we can
trivially extend any independent system by adding the equation
$t=x$. This gives $\is_+(4) \geq 4$ and $\is'_+(4) \geq 3$. For
chains the improvements are nontrivial. The following chain of
equations shows that $\dc(4) \geq 12$:
\begin{align*}
    xyz &= zxy,&
        x&=a, \ y=b, \ z=abab, \ t=a\\
    xyt &= txy,&
        x&=a, \ y=b, \ z=abab, \ t=abab\\
    xy xzy z &= z xzy xy,&
        x&=a, \ y=b, \ z=ab, \ t=abab\\
    xy xty t &= t xty xy,&
        x&=a, \ y=b, \ z=ab, \ t=ab\\
    xy xzty zt &= zt xzty xy,&
        x&=a, \ y=b, \ z=ab, \ t=1\\
    xz &= zx,&
        x&=a, \ y=b, \ z=1, \ t=ab\\
    xt &= tx,&
        x&=a, \ y=b, \ z=1, \ t=1\\
    xy &= yx,&
        x&=a, \ y=a, \ z=a, \ t=a\\
    x &= 1,&
        x&=1, \ y=a, \ z=a, \ t=a\\
    y &= 1,&
        x&=1, \ y=1, \ z=a, \ t=a\\
    z &= 1,&
        x&=1, \ y=1, \ z=1, \ t=a\\
    t &= 1,&
        x&=1, \ y=1, \ z=1, \ t=1.
\end{align*}}

The next theorem sums up the new bounds given in this section.

\begin{theorem}
$\dc_+(3) \geq 7$ and $\dc(4) \geq 12$.
\end{theorem}

\section{Lower Bounds}

In \cite{KaPl96} it is proved that $\is(n) = \Omega(n^4)$ and
$\is_+(n) = \Omega(n^3)$. The former is proved by a construction
that uses $n = 10m$ variables and gives a system of $m^4$ equations.
Thus $\is(n)$ is asymptotically at least $n^4/10000$. We present
here a slightly modified version of this construction. By
''reusing'' some of the unknowns we get a bound that is
asymptotically $n^4/1536$.

\begin{theorem}
If $n = 4m$, then $\is'(n) \geq m^2(m-1)(m-2)/6$.
\end{theorem}
\begin{proof}
We use unknowns $x_i, y_i, z_i, t_i$, where $1 \leq i \leq m$. The
equations in the system are
\begin{equation*}
    E(i,j,k,l): x_i x_j x_k  y_i y_j y_k  z_i z_j z_k  t_l
          = t_l x_i x_j x_k  y_i y_j y_k  z_i z_j z_k ,
\end{equation*}
where $i,j,k,l \in \{1, \dots, m\}$ and $i<j<k$. If $i,j,k,l \in
\{1, \dots, m\}$ and $i<j<k$, then
\begin{align*}
    x_r &= \begin{cases}
    ab, & \text{if $r \in \{i,j,k\}$} \\
    1, & \text{otherwise}
    \end{cases}
    \quad
    y_r = \begin{cases}
    a, & \text{if $r \in \{i,j,k\}$} \\
    1, & \text{otherwise}
    \end{cases}
    \\
    z_r &= \begin{cases}
    ba, & \text{if $r \in \{i,j,k\}$} \\
    1, & \text{otherwise}
    \end{cases}
    \quad
    t_r = \begin{cases}
    ababa, & \text{if $r=l$} \\
    1, & \text{otherwise}
    \end{cases}
\end{align*}
is not a solution of $E(i,j,k,l)$, but is a solution of all the
other equations. Thus the system is independent.
\end{proof}

The idea behind this construction (both the original and the
modified) is that
\begin{math}
    (ababa)^k = (ab)^k a^k (ba)^k
\end{math}
holds for $k<3$, but not for $k=3$. It was noted in \cite{Pl03} that
if we could find words $u_i$ such that
\begin{math}
    (u_1 \dots u_m)^k = u_1^k \dots u_m^k
\end{math}
holds for $k<K$, but not for $k=K$, then we could prove that $\is(n)
= \Omega(n^{K+1})$. However, it has been proved that such words do
not exist for $K \geq 5$ (see \cite{Ho01}), and conjectured that
such words do not exist for $K=4$.

{\allowdisplaybreaks For small values of $n$ it is better to use
ideas from the constructions showing that $\dc(3) \geq 7$ and
$\dc(4) \geq 12$. This gives $\is'(n) \geq (n^2 - 5n + 6)/2$ and
$\dc(n) \geq (n^2 + 3n - 4)/2$. The equations in the system are
\begin{equation*}
    x y  x z_i z_j y  z_i z_j = z_i z_j  x z_i z_j y  x y ,
\end{equation*}
where $i, j \in \{1, \dots, n-2\}$ and $i<j$. The equations in the
chain are
\begin{align*}
    x y z_k &= z_k x y ,\\
    x y  x z_k y  z_k &= z_k x z_k y  x y ,\\
    x y  x z_i z_j y  z_i z_j &= z_i z_j  x z_i z_j y  x y ,\\
    x z_k &= z_k x ,\\
    x y &= y x ,\\
    x &= 1 ,\\
    y &= 1 ,\\
    z_k &= 1 ,
\end{align*}
where $i, j \in \{1, \dots, n-2\}$, $i<j$ and $k \in \{1, \dots,
n-2\}$. Here we should first take the equations on the first row in
some order, then the equations on the second row in some order, and
so on.}

We conclude this section by mentioning a related question. It is
well known that any nontrivial equation on $n$ variables forces a
defect effect; that is, the values of the variables in any solution
can be expressed as products of $n-1$ words (see \cite{HaKa04} for a
survey on the defect effect). If a system has only periodic solutions,
then the system can be said to force a maximal defect effect, so
$\is'(n)$ is the maximal size of an independent system not doing
that. But how large can an independent system be if it forces only
the minimal defect effect, that is, the system has a solution in
which the variables cannot be expressed as products of $n-2$ words?
In \cite{KaPl96} it is proved that there are such systems of size
$\Omega(n^3)$ in free monoids and of size $\Omega(n^2)$ in free
semigroups. Again, no upper bounds are known.

\section{Concluding Remarks and Open Problems}

To summarize, we list a few fundamental open problems about systems
and chains of equations in free monoids.

\begin{enumerate}[{Question} 1:]
\item Is $\is(3)$ finite? \label{q1}
\item Is $\dc(3)$ finite? \label{q2}
\item Is $\is(n)$ finite for every $n$? \label{q3}
\item Is $\dc(n)$ finite for every $n$? \label{q4}
\end{enumerate}

A few remarks on these questions are in order. First we know that
each of these values is at most $\ub$. Second, if the answer to any
of the questions is ''yes'', a natural further question is: What is
an upper bound for this value, or more sharply, what is the best
upper bound, that is, the exact value? For the lower bounds the best
what is known, according to our knowledge, is the following
\begin{enumerate}[{Question} 1:]
\item $\is(3) \geq 3$,
\item $\dc(3) \geq 7$,
\item $\is(n) = \Omega(n^4)$,
\item $\dc(n) = \Omega(n^4)$.
\end{enumerate}
A natural sharpening of Question \ref{q3} (and \ref{q4}) asks
whether these values are exponentially bounded.

A related question to Question \ref{q1} is the following amazing
open problem from \cite{CuKa83} (see, e.g., \cite{Cz08} and
\cite{CzKa07} for an extensive study of it):

\begin{enumerate}[{Question} 5:]
\item Does there exist an independent system of three equations with
    three unknowns having a nonperiodic solution?
\end{enumerate}

As a summary we make the following remarks. As we see it, Question
\ref{q3} is a really fundamental question on word equations or even
on combinatorics on words as a whole. Its intriguity is revealed by
Question \ref{q1}: we do not know the answer even in the case of
three unknowns. This becomes really amazing when we recall that
still the best known lower bound is only 3!

To conclude, we have considered equations over word monoids and
semigroups. All of the questions can be stated in any semigroup, and
the results would be different. For example, in commutative monoids
the compactness property (Theorem \ref{thm:compactness}) holds, but
in this case the value of the maximal independent system of
equations is $\ub$ (see \cite{KaPl96}).

\section*{Note added on June 9, 2015}

When writing the original version of this article,
we were not aware of any previous research on increasing chains.
However, they were defined and studied already in 1999
by Honkala \cite{Ho99} (in the case of free monoids).
They were called descending chains in that article.
Decreasing chains were called ascending chains
and Theorem \ref{thm:cp_dc} was proved in the case of free monoids.
Most of the paper was devoted to
descending chains and test sets.

The following conjecture, which we state here using our notation,
was given in \cite{Ho99}:
In free monoids $\ic(n) \leq \ub$ for all $n$.
This appears to be a very interesting and difficult problem.
Proofs of Ehrenfeucht's conjecture are ultimately based on the fact that
ideals in polynomial rings satisfy the ascending chain condition.
As pointed out in \cite{Ho99},
the same is not true for the descending chain condition,
so the above conjecture
could be expected to be significantly more difficult to prove
than Ehrenfeucht's conjecture was.
Of course,
if $\dc(n) < \ub$, then $\ic(n) = \dc(n)$ by Theorem \ref{thm:basic}.


\end{document}